\documentclass{amsart}
\usepackage{amsmath}
\usepackage{amsthm}
\usepackage{amssymb}
\usepackage[all,cmtip]{xy}
\usepackage{tikz-cd}
\usetikzlibrary{matrix}
\usepackage{hyperref}

\newcommand{\longto}{\longrightarrow}

\makeatletter
\newcommand*\rel@kern[1]{\kern#1\dimexpr\macc@kerna}
\newcommand*\widebar[1]{%
  \begingroup
  \def\mathaccent##1##2{%
    \rel@kern{0.8}%
    \overline{\rel@kern{-0.8}\macc@nucleus\rel@kern{0.2}}%
    \rel@kern{-0.2}%
  }%
  \macc@depth\@ne
  \let\math@bgroup\@empty \let\math@egroup\macc@set@skewchar
  \mathsurround\z@ \frozen@everymath{\mathgroup\macc@group\relax}%
  \macc@set@skewchar\relax
  \let\mathaccentV\macc@nested@a
  \macc@nested@a\relax111{#1}%
  \endgroup
}
\makeatother

\DeclareMathOperator\Aut{Aut}

\DeclareMathOperator\Hom{Hom}

\DeclareMathOperator\Spec{Spec}

\DeclareMathOperator\rmO{O}
\DeclareMathOperator\SO{SO}
\DeclareMathOperator\Spin{Spin}

\DeclareMathOperator\Res{Res}

\DeclareMathOperator\Nm{Nm}

\DeclareMathOperator\Br{Br}

\DeclareMathOperator\sgn{sgn}
\DeclareMathOperator\tr{tr}
\DeclareMathOperator\disc{disc}

\DeclareMathOperator\inv{inv}

\def\bQ{{\mathbf{Q}}} \def\bZ{{\mathbf{Z}}} 
\def\bF{{\mathbf{F}}} \def\bG{{\mathbf{G}}} \def\bR{{\mathbf{R}}}
\def\bC{{\mathbf{C}}}

\def\cO{{\mathcal{O}}}  
  \def\cL{{\mathcal{L}}} 
  
 \def\cD{{\mathcal D}} 
  \def\cB{{\mathcal{B}}}
\def\cS{{\mathcal S}}

  \def\rH{{\rm H}}

\def\fm{{\mathfrak{m}}}

\def\id{{\rm id}}

\def\rho{\varrho}
\def\sp{{\rm sp}}
\def\un{{\rm un}}

\theoremstyle{plain}
\newtheorem{theorem}{Theorem}[section]
\newtheorem{corollary}[theorem]{Corollary}
\newtheorem*{corollary*}{Corollary}
\newtheorem{lemma}[theorem]{Lemma}
\newtheorem{bigthm}{Theorem}
\newtheorem{proposition}[theorem]{Proposition}
\theoremstyle{definition}
\newtheorem{definition}[theorem]{Definition}
\newtheorem{example}[theorem]{Example}

\newtheorem{remark}[theorem]{Remark}

\newcommand{\Creciprocal}{(C1)}
\newcommand{\Creal}{(C2)}
\newcommand{\Csquares}{(C3)}

\begin{document}

\title[Unimodular lattices and equivariant Witt groups]{Automorphisms of even unimodular lattices \\  and equivariant Witt groups}

\author{Eva Bayer-Fluckiger}

\author{Lenny Taelman}

\begin{abstract}
We characterize the irreducible polynomials that occur as a characteristic polynomial of an automorphism of an even unimodular lattice of given signature, generalizing a theorem of Gross and McMullen. As part of the proof, we give a general criterion in terms of Witt groups for a bilinear form equipped with an action of a group $G$ over a discretely valued field to contain a unimodular $G$-stable lattice. 
\end{abstract}

\maketitle

\section{Introduction}

\subsection{Statement of the main result}

Let $r$ and $s$ be non-negative integers, and let $\Lambda$ be an even unimodular lattice of signature $(r,s)$. Such a lattice exists if and only if $r\equiv s\bmod 8$, and it is uniquely determined by $(r,s)$ if $r$ and $s$ are non-zero. Let $\alpha \in \SO(\Lambda)$ and denote by $S$ its characteristic polynomial. Denote by $m(S)$ be the number of complex roots $z$ of $S$ with $|z|>1$ (counted with multiplicity). Gross and McMullen \cite{GrossMcMullen02} show that if $S$ has no linear factor, then $S$ satisfies
\begin{enumerate}
\item[\Creciprocal]  $S$ is reciprocal (\emph{i.e.}~$t^{r+s}S(1/t)=S(t)$),
\item[\Creal]  $m(S) \leq r$, $m(S)\leq s$ and $m(S)\equiv r \equiv s \pmod 2$,
\item[\Csquares] $|S(1)|$, $|S(-1)|$ and $(-1)^{\frac{r+s}{2}}S(1)S(-1)$ are squares.
\end{enumerate}
They speculate that for a monic irreducible $S\in\bZ[t]$ of degree $r+s$, the above conditions may be \emph{sufficient} for the existence of an even unimodular lattice $\Lambda$ of signature $(r,s)$ and an $\alpha \in\SO(\Lambda)$ with characteristic polynomial $S$. They show that this is indeed so if  $|S(1)|=|S(-1)|=1$ (a different proof was given by the first author in \cite{Bayer02}). The main result of this paper confirms their speculation:

\begin{bigthm}\label{bigthm:even-unimodular-charpol}
Let $r$, $s$ be non-negative integers  satisfying $r\equiv s \bmod 8$. Let $P$ be a monic irreducible polynomial, and let $S$ be a power of $P$. Assume that $S$ has degree $r+s$, and that it satisfies {\rm \Creciprocal}, {\rm \Creal}, and {\rm \Csquares}.
Then there exists an even unimodular lattice $\Lambda$ of signature $(r,s)$ and an $\alpha \in \SO(\Lambda)$ with characteristic polynomial $S$.
\end{bigthm}

Gross and McMullen have observed \cite[Prop.~5.2]{GrossMcMullen02} that the theorem does not hold for  $S$ that have distinct irreducible factors.  There are results by the first author on reducible characteristic polynomials of \emph{definite} even unimodular lattices \cite{Bayer84,Bayer87}, but in general it is not even clear what to conjecture. 

As in \cite[\S~8]{GrossMcMullen02} and \cite[Thm~3.4]{McMullen02}, Theorem \ref{bigthm:even-unimodular-charpol} together with the Torelli theorem for K3 surfaces has the following immediate consequence.

\begin{corollary*}\label{cor:K3}
Let $S\in \bZ[t]$ be a monic, irreducible, reciprocal polynomial of degree $22$. Assume that $S$ has precisely $1$ root $z\in \bC$ with $|z|>1$, and assume that $|S(1)|$, $|S(-1)|$  and $-S(1)S(-1)$ are squares. Then there exists a complex analytic K3 surface $X$ and an $\alpha\in \Aut X$, such that the characteristic polynomial of $\alpha^\ast$ on $\rH^2(X,\bZ)$ is $S$. \qed
\end{corollary*}


\subsection{Strategy of proof}\label{intro:strategy}
We sketch the proof of Theorem \ref{bigthm:even-unimodular-charpol}, and refer to the main body of the paper for the details. Assume for simplicity that $S$ is irreducible.

Let $E=\bQ[t]/(S)$ and let $\alpha \in E$ be the class of $t$. By condition \Creciprocal, the field $E$ comes with an involution $\sigma$ characterized by $\sigma(\alpha) = \alpha^{-1}$. Let $E_0 \subset E$ be the fixed field under $\sigma$. For every $\lambda \in E_0^\times$ we have a symmetric $\bQ$-bilinear form
\[
	b_\lambda \colon E\times E \to \bQ,\, (x,y) \mapsto \tr_{E/\bQ} (\lambda x \sigma(y)),
\]
and $\alpha$ acts on $(E,b_\lambda)$ by isometries.  Condition \Creal~guarantees that for suitable $\lambda$ the form $b_\lambda$ will have signature $(r,s)$. It we can moreover choose $\lambda$ in such a way that $b_\lambda$ contains an $\alpha$-stable even unimodular lattice, then the theorem follows.

Gross and McMullen show that if $|S(1)|=|S(-1)|=1$, then for suitable $\lambda$  the lattice $\Lambda$ can be taken to be a fractional $\cO_E$-ideal. In general, however,  an $\alpha$-stable lattice $\Lambda$ will only be a module over the order $\bZ[\alpha^{\pm 1}] \subset \cO_E$, severely complicating the situation at places where $\bZ[\alpha^{\pm 1}]$ fails to be maximal.

We take a somewhat different approach. We first show that for every $p$ there are $\lambda\in (\bQ_p\otimes E_0)^\times$  for which there exists a unimodular $\alpha$-stable $\bZ_p$-lattice $\Lambda_p$ in $(\bQ_p \otimes E, b_\lambda)$. This is done by showing that under condition \Csquares~we can make the obstruction class in an `equivariant Witt group' over $\bF_p$ vanish, see \S~\ref{intro:local} below. A careful analysis at the prime $2$ shows that one can moreover arrange for the $\bZ_2$-lattice $\Lambda_2$ to be even.

Thus, we can locally solve the problem. To finish the proof, it suffices to show that there exists a global $\lambda$  satisfying the local conditions (at finite and infinite places). This is the case if and only if a certain  local-global obstruction class in $\bZ/2\bZ$ vanishes. It appears to be hard to compute the individual contributions to this obstruction directly, but we observe that the obstruction class can be realized as the difference of two obstruction classes coming from global objects, and hence that it must vanish. See \S~\ref{sec:the-proof} for more details.

\subsection{Local obstructions in equivariant Witt groups}\label{intro:local}
Let $K$ be a field and $G$ a group. We call a \emph{$G$-bilinear form} a finite-dimensional $K$-linear representation $V$ of $G$ equipped with a non-degenerate $G$-invariant symmetric bilinear form $V\times V \to K$. We say $V$ is \emph{neutral} if there exists a $G$-stable subspace $X\subset V$ with $X=X^\perp$. The \emph{equivariant Witt group} $W_G(K)$ is the quotient of the Grothendieck group of $G$-bilinear forms by the subgroup generated by the classes of neutral forms. If the group $G$ is trivial, then $W_G(K)$ coincides with the usual Witt group $W(K)$. In \S~\ref{sec:equivariant-witt} we develop the basic properties of these equivariant Witt groups, in the slightly more general (and more natural) context of $\epsilon$-symmetric bilinear forms over a ring with involution.

Now let $K$ be the fraction field of a discrete valuation ring $\cO_K$ with uniformizer $\pi$ and residue field $k$. We call an $\cO_K$-lattice $\Lambda$ in a $K$-bilinear form $V$ \emph{almost unimodular} if
\[
	\pi\Lambda^\vee \subset \Lambda \subset \Lambda^\vee,
\]
If $V$ is a $G$-bilinear form over $K$ and if $\Lambda \subset V$ is almost unimodular and $G$-stable, then the quotient $\Lambda^\vee/\Lambda$ equipped with the form
\[
	\frac{\Lambda^\vee}{\Lambda} \times \frac{\Lambda^\vee}{\Lambda}  \overset{b}{\longto} 
	\frac{\pi^{-1}\cO_K}{\cO_K} \overset{\pi}{\longto} k
\]
is a $G$-bilinear form over $k$. 

We call a $G$-bilinear form $V$ \emph{bounded} if $V$  contains a $G$-stable $\cO_K$-lattice, or equivalently, if the closure of the image of $G$ in $\rmO(V)$ is compact.
We denote by $W^b_G(K)$ the subgroup of $W_G(K)$ generated by the bounded forms. 

\begin{bigthm}\label{bigthm:reduction}Let $K$ be the fraction field of a discrete valuation ring $\cO_K$ with  residue field $k$, and fix a uniformizer $\pi$ of $\cO_K$. Let $G$ be a group. Then
\begin{enumerate}
\item every bounded $G$-bilinear form $V$ contains an almost unimodular $G$-stable lattice $\Lambda$,
\item the class of $[\Lambda^\vee/\Lambda]$  in $W_G(k)$ only depends on the class of $V$ in $W_G(K)$,
\item the map $\partial\colon W^b_G(K) \to W_G(k)$ given by $[V] \mapsto [\Lambda^\vee/\Lambda]$ is a homomorphism,
\item a $G$-bilinear form $V$ contains a  $G$-stable unimodular  $\cO_K$-lattice if and only if $V$ is bounded and $\partial [V]= 0$ in $W_G(k)$.
\end{enumerate}
\end{bigthm}

As in our treatment of the equivariant Witt groups, we show this in a slightly more general context, see Theorem \ref{thm:B}. This theorem can be seen as an equivariant version of a theorem of Springer \cite{Springer55} (but see Remark \ref{rmk:Springer-fails}), or of the theory of discriminant forms of Nikulin \cite{Nikulin79}. The result is without doubt well-known to the experts, but for lack of proper reference we include a complete proof in sections \ref{sec:equivariant-witt} and \ref{sec:reduction}. See also  \cite{Serre17},  \cite{Thompson86},  and \cite{Stoltzfus77} where similar ideas are used.

 In case the residue field $k$ is finite and $G$ is the infinite cyclic group, then the group $W_G(k)$ and the  map $\partial$ can be made reasonably explicit. We use this in the proof of Theorem \ref{bigthm:even-unimodular-charpol} to show (using the notation of \S~\ref{intro:strategy}) that there is no local obstruction to finding a $\lambda \in E_0^\times$ and a unimodular  $\Lambda \subset (E,b_\lambda)$, stable under $\alpha$. See Propositions \ref{prop:at-p} and \ref{prop:local-unimodular-existence-at-2}.

If $K=\bQ_2$, then we use the spinor norm to refine the theorem into a necessary and sufficient condition for a $G$-bilinear form to contain a $G$-invariant \emph{even} unimodular $\bZ_2$-lattice; see Theorem \ref{thm:evenness}, and its application in Proposition \ref{prop:at-2}.

The full strength of Theorem \ref{bigthm:reduction} is not needed for the application in Theorem \ref{bigthm:even-unimodular-charpol}, and one could make do with a few ad hoc computations in the spirit of the proof of Theorem \ref{bigthm:reduction}. However, we have included it since it helps keeping the argument organized, and since we believe Theorem \ref{bigthm:reduction} may be a useful tool in working with `$G$-lattices', playing a role similar to Nikulin's theory for lattices without group action.

\subsection*{Acknowledgements}We want to thank Asher Auel, Simon Brandhorst, Curt McMullen, Jean-Pierre Serre and Xun Yu for enlightening discussions and comments on an earlier draft. The second author is supported by a grant of the Netherlands Organisation for Scientific Research (NWO).

\section{Conventions}

The \emph{determinant} of a non-degenerate symmetric bilinear form is $\det b \in K^\times/(K^\times)^2$ is the determinant of any Gram matrix of $b$. The \emph{discriminant} $\disc b\in K^\times/(K^\times)^2$ is defined as
\[
	\disc b := (-1)^{\frac{n(n-1)}{2}} \det b,
\]
where $n$ is the dimension of $V$. The discriminant of a hyperbolic form is trivial.

\section{The equivariant Witt group}\label{sec:equivariant-witt}

Let $K$ be a field and $A$ an (associative, unital) $K$-algebra. Let $\sigma\colon A\to A$ be an involution. We will refer to the algebra with involution $(A,\sigma)$ simply by $A$. Fix an $\epsilon$ in $\{\pm 1\}$.

\begin{definition}An \emph{$(A,\epsilon)$-bilinear form} is a pair $(V,b)$ consisting of an $A$-module $V$ of finite $K$-dimension and a non-degenerate $K$-bilinear form $b\colon V\times V\to K$ satisfying
\begin{enumerate}
\item $b(ax,y) = b(x,\sigma(a)y)$ for all $a\in A$ and $x,y\in V$,
\item $b(x,y) = \epsilon b(y,x)$ for all $x,y\in V$.
\end{enumerate}
\end{definition}
We will sometimes suppress $b$ and and denote an $(A,\epsilon)$-bilinear form by $V$.

\begin{example}In the proof of Theorem \ref{bigthm:even-unimodular-charpol} only pairs $(A,\epsilon)$ of the following kind will be considered. Let $G$ be a group, $A = K[G]$, and $\sigma$ the involution of $A$ satisfying $\sigma(g)=g^{-1}$ for all $g\in G$. In this case, an $(A,+1)$-bilinear form is the same as a non-degenerate symmetric bilinear form $b\colon V\times V \to K$ together with an action $\rho\colon G\to \rmO(V,b)$. We will also refer to such a triple $(V,b,\rho)$  as a \emph{$G$-bilinear form}.
\end{example}

\begin{definition}Let $V$ be an $(A,\epsilon)$-bilinear form. If $X\subset V$ is a sub-$A$-module, then its \emph{orthogonal} $X^\perp = \{ x\in V \mid b(x,V) = 0 \}$ is also a sub-$A$-module. A \emph{lagrangian} is a sub-$A$-module $X\subset V$ satisfying $X=X^\perp$. We say that $V$ is \emph{neutral} if it has a lagrangian. 
The group  $W_A^\epsilon(K)$ is defined as the quotient of the Grothendieck group of $(A,\epsilon)$-bilinear forms by the subgroup generated by the neutral forms.
\end{definition}

If $A=K[G]$ and $\epsilon=1$ then we will write $W_G(K)$ for $W_A^\epsilon(K)$.
 
\begin{remark}
If $A=K$ and $\epsilon=1$ then $W_A^\epsilon(K)$ coincides with the usual Witt group of bilinear forms $W(K)$. Indeed, if  $X\subset V$ satisfies $X = X^\perp$ then 
 $V$ is isomorphic to $X \oplus X^\vee$  equipped with the obvious bilinear form, and $V$ is hyperbolic  (or metabolic in characteristic $2$).
\end{remark}

\begin{remark} \label{rmk:no-dual-lagrangian}
 The  decomposition $V\cong X \oplus X^\vee$ in the preceding remark is not canonical, and in general, if $X$ is a lagrangrian in an $(A,\epsilon)$-bilinear form $V$ then the natural short exact sequence of $A$-modules
 \[
 	0 \longto X \longto V \longto X^\vee \longto 0
\]
need not split.
\end{remark}

 If $V=(V,b)$ is an $(A,\epsilon)$-bilinear form, then we denote by $V(-1)$ the scaled $(A,\epsilon)$-bilinear form $(V,-b)$.

\begin{lemma}$V\oplus V(-1)$ is neutral.
\end{lemma}

\begin{proof}The submodule $\{ (x,x) \mid x \in V\}$ is a lagrangian.
\end{proof}

\begin{corollary}
Every element of $W_A^\epsilon(K)$ is of the form $[V]$ for some $V$. \qed
\end{corollary}

\begin{lemma}\label{lemma:totally-isotropic-simplification}
If $X\subset V$ is a sub-$A$-module and satisfies $X\subset X^\perp$, then 
$[V]=[X^\perp/X]$ in $W_G(K)$.
\end{lemma}

\begin{proof}
The submodule $\{ (x,\bar x) \mid x \in X^\perp \}$ is a lagrangian in $V(-1) \oplus X^\perp/X$.
\end{proof}

\begin{proposition}\label{prop:neutral-iff-zero}
$[V]=0$ in $W_A^\epsilon(K)$ if and only if $V$ is neutral.
\end{proposition}

\begin{proof}If $V$ is neutral then $[V]=0$ by definition of $W_A^\epsilon(K)$. 

Conversely, assume that $[V]=0$. Then there exists a $W$ and a neutral $N$ such that $V\oplus W \cong N \oplus W$. Adding a summand $W(-1)$ we find that there exists a neutral $M$ such that
$V \oplus M$ is neutral.

Let $2n$ be the dimension of $V$, and $2m$ the dimension of $M$.
Let $X\subset V\oplus M$, and $Y\subset M$ be lagrangians.  Consider the submodule
\[
	S := X \cap (V\times Y) \subset V \oplus M.
\]
We have  $X+(V\times Y) \subset (X\cap Y)^\perp$, so that we find
\[
	\dim X + \dim (V\times Y)  \leq \dim S + \dim (X\cap Y)^\perp.
\]
Expressing everything in terms of $n$, $m$, $\dim S$ and $\dim (X\cap Y)$ gives
\[
	(n+m) + (2n+m) \leq \dim S + (2n+2m-\dim (X\cap Y)),
\]
and hence $\dim S - \dim(X\cap Y) \geq n$.
Now the subspace $Z:= \pi_V(S) \subset V$ is a totally isotropic sub-$A$-module. The kernel of the projection
$S\to Z$ is $X\cap Y$, so we have
\[
	\dim Z = \dim S - \dim (X\cap Y) \geq n,
\]
and therefore $Z$ is a lagrangian. 
\end{proof}

The following proposition is analogous to the diagonalizability of quadratic forms over fields.

\begin{proposition}\label{prop:simple-generators}
The group $W_A^\epsilon(K)$ is generated by the classes of $(A,\epsilon)$-bilinear forms $(V,b)$ 
with $V$ a simple $A$-module.
\end{proposition}

\begin{proof}
Let $W^{\mathrm{ss}} \subset W_A^\epsilon(K)$ be the subgroup generated by the classes of forms on simple $A$-modules. Let $[V]$ be an element of the complement $ W_A^\epsilon(K) \setminus W^{\mathrm{ss}}$ with $\dim_K V$ minimal. 

Since $V$ is not simple, it contains a proper submodule $W$. Consider the submodule  $X = W \cap W^\perp$ of $V$. Either $X=\{0\}$ and then $[V]=[W] + [W^\perp]$, or $X\neq \{0\}$ and then  $[V]=[X^\perp/X]$. In both cases we obtain a contradiction with the minimality of $\dim V$.
\end{proof}

\begin{corollary}
Every class in  $W_A^\epsilon(K)$ is represented by an $(A,\epsilon)$-bilinear form whose underlying $A$-module is semi-simple.\qed 
\end{corollary}

If $M$ is a simple $A$-module, then we denote by 
$W_A^\epsilon(K,M)$ the subgroup of $W_A^\epsilon(K)$ generated by the classes of $(A,\epsilon)$-bilinear forms $(M,b)$.

\begin{theorem}\label{thm:isotypical-decomposition}
$W_A^\epsilon(K) = \oplus_M W_A^\epsilon(K,M)$ where $M$ ranges over the isomorphism classes of simple $A$-modules.
\end{theorem}

\begin{proof}
By Proposition \ref{prop:simple-generators} the map $\oplus_M W_A^\epsilon(K,M) \to W_A^\epsilon(K)$
is surjective. Let $([V_M])_M$ be an element of the kernel. Then by Proposition \ref{prop:neutral-iff-zero} there is a lagrangian $X \subset \oplus_M V_M$, and since $\Hom_A(M_1,M_2) =  0$ whenever $M_1\not\cong M_2$ we have that $X$ decomposes as $X=\oplus_M X_M$. For every $M$ the submodule $X_M \subset V_M$ is a lagrangian, and we conclude that the map is injective.
\end{proof}

\section{Equivariant forms over discrete valuation rings}\label{sec:reduction}

\subsection{Statement of the result}

Throughout this section, $\cO_K$ is a discrete valuation ring with field of fractions $K$, residue field $k$, and uniformizer $\pi$. Let $(A,\sigma)$ be an $\cO_K$-algebra with involution and write $A_K$ and $A_k$ for $A\otimes_{\cO_K} K$ and $A\otimes_{\cO_K} k$ respectively. To lighten notation, we will write $W_A^\epsilon(K)$ for $W_{A_K}^\epsilon(K)$ and similarly $W_A^\epsilon(k)$ for $W_{A_k}^\epsilon(k)$.

We will be particularly interested in the case $A=\cO_K[G]$ with $\epsilon=1$, in which case we have $W_A^\epsilon(K)=W_G(K)$ and $W_A^\epsilon(k)=W_G(k)$.

\begin{definition}
 An $A$-\emph{lattice} in an $(A_K,\epsilon)$-bilinear form $V$ is a sub-$A$-module $\Lambda$ which is finitely generated as an $\cO_K$-module and satisfies $K\Lambda = V$. If $\Lambda$ is an $A$-lattice, then so is its dual
 \[
	\Lambda^\vee := \{ x\in V \mid b(x,\Lambda) \subset \cO_K \}.
\]
We say that $\Lambda$ is is \emph{almost unimodular} if $\pi\Lambda^\vee \subset \Lambda \subset \Lambda^\vee$, and that it is \emph{unimodular} if $\Lambda = \Lambda^\vee$.
\end{definition}

If $\Lambda$ is almost unimodular, then $\Lambda^\vee/\Lambda$ equipped with the pairing
\[
	\Lambda^\vee/\Lambda \times \Lambda^\vee/\Lambda
	\overset{b}\longto \frac{ \pi^{-1}\cO_K }{\cO_K} \overset{\pi} \longto k
\]
is an $(A_k,\epsilon)$-bilinear form.

\begin{definition}
We say that an $(A_K,\epsilon)$-bilinear form is \emph{bounded} if it contains an $A$-lattice. We denote by $W_{A}^{\epsilon,b}(K)$ the subgroup of $W_{A}^\epsilon(K)$ generated by the classes of bounded forms.
\end{definition}

In this section we prove the following theorem, which for $A=\cO_K[G]$ and $\epsilon=1$ coincides with Theorem \ref{bigthm:reduction}. 

\begin{theorem}\label{thm:B}$~$
\begin{enumerate}
\item Every bounded $(A_K,\epsilon)$-bilinear form contains an almost unimodular $A$-lattice $\Lambda$,
\item the class of $\Lambda^\vee/\Lambda$ in $W_{A}^\epsilon(k)$ only depends on the class of $V$ in $W_{A}^\epsilon(K)$,
\item the map $\partial\colon W_{A}^{\epsilon,b}(K) \to W_{A}^\epsilon(k)$ given by $[V] \mapsto \Lambda^\vee/\Lambda$ is a homomorphism,
\item $V$ contains a unimodular $A$-lattice if and only if $V$ is bounded and $\partial[V]=0$ in $W_{A}^\epsilon(k)$.
\end{enumerate}
\end{theorem}

We follow closely the proof of \cite[\S~6.1]{Scharlau85}, which treats the case $A=\cO_K$, $\epsilon=1$.

\subsection{Torsion forms and proof of Theorem \ref{thm:B}}

\begin{definition}An \emph{$(A,\epsilon)$-bilinear torsion form} $M$ is a pair $(M,b)$ consisting of 
\begin{enumerate}
\item an $A$-module $M$ of finite finite length over $\cO_K$,
\item a  non-degenerate $\cO_K$-bilinear map $b\colon M\times M \to K/\cO_K$ satisfying $b(x,y)=\epsilon b(y,x)$ and $b(ax,y) = b (x,\sigma(a)y)$ for all $x,y\in M$ and $a\in A$.
\end{enumerate}
A \emph{lagrangian} in $M$ is a sub-$A$-module $X\subset M$ with $X=X^\perp$. We say that $M$ is \emph{neutral} if $M$ has a lagrangian.
\end{definition}

If $V$ is an $(A_K,\epsilon)$-bilinear form and $\Lambda \subset V$ an $A$-lattice satisfying $\Lambda \subset \Lambda^\vee$, then $M := \Lambda^\vee/\Lambda$ is naturally an $(A,\epsilon)$-bilinear torsion form (sometimes called the discriminant form of $\Lambda$). The  $A$-lattices $\Lambda'$ satisfying
\[
	\Lambda \subset \Lambda'\subset \Lambda'^\vee \subset \Lambda^\vee
\]
are in bijection with the totally isotropic sub-$A$-modules $M'\subset M$,  and $\Lambda'$ is almost unimodular (resp.~unimodular) if and only $\pi M'^\perp \subset M'$ (resp.~$M'$ is a lagrangian).

\begin{proposition}
Let $M$ be an $(A,\epsilon)$-bilinear torsion form. Then every maximal  totally isotropic sub-$A$-module $U\subset M$ satisfies $\pi U^\perp \subset U$.
\end{proposition}

\begin{proof}
Let $U \subset M$ be a maximal  totally isotropic submodule. Let $t$ be the largest integer such that $\pi^{t-1} U^\perp/U\neq 0$. Assume $t\geq 2$ and set $V:=\pi^{t-1}U^\perp + U$. Then $V$ is a sub-$A$-module, and we claim that it is totally isotropic. Indeed, if $x\in U^\perp$ and $y\in U$ then
\[
	b(\pi^{t-1} x+ y , \pi^{t-1}x+y) = b(\pi^t x, \pi^{t-2} x) = 0,
\]
since $\pi^t x \in U$. This is a contradiction with the maximality of $U$.
\end{proof}

\begin{corollary}
Let $V$ be an $(A_K,\epsilon)$-bilinear form. Then any maximal $A$-lattice in $V$ is almost unimodular.\qed
\end{corollary}

\begin{corollary}\label{cor:almost-unimodular-lattices-exist}
Let $V$ be a bounded $(A,\epsilon)$-bilinear form over $K$. Then $V$ has an almost unimodular $A$-lattice.\qed
\end{corollary}

We denote by $WT_A^\epsilon(\cO_K)$ the Witt group of $(A,\epsilon)$-bilinear torsion forms, defined as the quotient of the Grothendieck group of $(A,\epsilon)$-bilinear torsion forms by the subgroup generated by the classes of the neutral forms. 

If $V$ is an $A_k$-bilinear form, then the composition
\[
	V\times V \longto k \overset{\pi^{-1}}{\longto} \frac{K}{\cO_K}
\]
makes $V$ into an $(A,\epsilon)$-bilinear torsion form over $\cO_K$. This defines a homomorphism $W_A^\epsilon(k) \to WT_A^\epsilon(\cO_K)$.

\begin{proposition}\label{prop:residue-field-to-torsion}
The map $W_A^\epsilon(k) \to WT_A^\epsilon(\cO_K)$ is an isomorphism.
\end{proposition}

\begin{proof}For $A=\cO_K=\bZ_p$, this is \cite[5.1.5]{Scharlau85}.  The argument carries over to our setting, we repeat it for the convenience of the reader. 

For a non-zero finite length $\cO_K$-module $M$, we call the smallest integer $t$ such that $\pi^t M = 0 $ but $\pi^{t-1}M \neq 0$ the  
 \emph{exponent} of $M$. 
 
Let $M$ be a $(A,\epsilon)$-bilinear torsion form of exponent $t\geq 2$. Then $U:=\pi^{t-1} M$ is totally isotropic, and the $(A,\epsilon)$-bilinear form $M' := U^\perp/U$ has exponent $<t$ and we have $[M]=[M']$ in $WT_A(\cO_K)$.  Repeating this process, we find a canonical $(A,\epsilon)$-bilinear torsion form $M^\dagger$ of exponent $1$ (\emph{i.e.} killed by $\pi$). 

Note that $(M\oplus N)^\dagger=M^\dagger \oplus N^\dagger$. Also, if $M$ is neutral with lagrangian $X\subset M$ then
\[
	X' := \frac{X\cap U^\perp}{X\cap U} \subset \frac{U^\perp}{U} = M'
\]
is lagrangian, since 
\[
	(X\cap U^\perp)^\perp \cap U^\perp = (X+U) \cap U^\perp = (X\cap U^\perp) + U
\]
and hence $X'$ is its own orthogonal in $U^\perp/U$. 
This shows that $M \mapsto M^\dagger$ induces a two-sided inverse to $W_A^\epsilon(k) \to WT_A(\cO_K)$.
\end{proof}

\begin{lemma}\label{lemma:independence-of-lattice}
If $\Lambda_0$ and $\Lambda_1$ are  $A$-lattices in $V$ satisfying $\Lambda_0\subset \Lambda_0^\vee$ and $\Lambda_1\subset \Lambda_1^\vee$, then
$[\Lambda_0^\vee/\Lambda_0] =[\Lambda_1^\vee/\Lambda_1]$ in $WT_A^\epsilon(\cO_K)$. 
\end{lemma}

\begin{proof}
Passing to the intersection, we may without loss of generality assume $\Lambda_0\subset \Lambda_1$. Let $M=\Lambda_0^\vee/\Lambda_1$. Then the image $U$ of $\Lambda_0$ in $M$ is totally isotropic, and $U^\perp/U \cong \Lambda_1^\vee/\Lambda_1$.
The same argument as in the proof of \ref{lemma:totally-isotropic-simplification} shows that $[M]=[U^\perp/U]$ and hence $[\Lambda_0^\vee/\Lambda_0] =[\Lambda_1^\vee/\Lambda_1]$ in $WT_A(\cO_K)$.
\end{proof}

Let $b\colon V\times V\to K$ be a non-degenerate bilinear form. For a sub-$\cO_K$-module $L$ of $V$ (not necessarily a lattice) we define
\[
	L^\vee := \{ x \in V \mid b(x,L) \subset \cO_K \}.
\]
The main properties of the dual of a lattice carry over to this generality:

\begin{lemma}\label{lemma:dual-prop}
For all sub-$\cO_K$-modules $L$, $M$ of a bilinear form $V$ we have 
\begin{enumerate} 
\item if $L$ is a sub-$K$-vector space, then $L^\vee = L^\perp$,
\item $(L^\vee)^\vee = L$,
\item $(L+M)^\vee = L^\vee \cap M^\vee$,
\item $(L \cap M)^\vee = L^\vee + M^\vee$.
\end{enumerate}
\end{lemma}

\begin{proof}
For (i), note that $b(x,\lambda y) \in \cO_K$ for all $\lambda \in K$ implies $b(x,y)=0$.

The second assertion is clear if $L$ is a sub-$K$-vector space, or an $\cO_K$-lattice. For a general $L$, note that there exist unique sub-$K$-vector spaces $V_0$ and $V_1$ such that $V_0 \subset L \subset V_1$
and $L/V_0$ is a lattice in $V_1/V_0$. The  dual $L^\vee$ satisfies $V_1^\perp \subset L^\vee \subset V_0^\perp$, with $L^\vee/V_1^\perp$ the dual lattice of $L/V_0$ under the perfect pairing
\[
	V_1/V_0 \times V_0^\perp/V_1^\perp \to K
\]
induced by $b$. Consequently, the double dual of $L$ coincides with $L$.

The third assertion is immediate, and the final one follows from the third using the double dual statement in (ii).
\end{proof}

\begin{lemma}\label{lemma:vanish-on-neutrals}
If $V$ is a neutral $A_K$-bilinear form, and if $\Lambda \subset V$ is an $A$-lattice satisfying $\Lambda \subset \Lambda^\vee$, then $\Lambda^\vee/\Lambda$ is a neutral torsion $A$-bilinear form.
\end{lemma}

\begin{proof}
Let $X\subset V$ be a lagrangian, and set
\[
	U := \frac{X\cap \Lambda^\vee}{X\cap \Lambda} \subset \frac{\Lambda^\vee}{\Lambda}.
\]
Then  using the Lemma \ref{lemma:dual-prop}, we see that $U^\perp \subset \Lambda^\vee/\Lambda$ satisfies
\[
	U^\perp = \frac{(X^\perp + \Lambda) \cap \Lambda^\vee}{(X^\perp + \Lambda) \cap \Lambda}
	= \frac{X\cap \Lambda^\vee}{X\cap \Lambda} = U,
\]
and we conclude that $U$ is a lagrangian in $\Lambda^\vee/\Lambda$.
\end{proof}

We are now ready to prove the main result of this section.

\begin{proof}[Proof of Theorem \ref{thm:B}]
Part (i) is Corollary \ref{cor:almost-unimodular-lattices-exist}. By Lemmas \ref{lemma:independence-of-lattice} and \ref{lemma:vanish-on-neutrals} the map $[V] \mapsto [\Lambda^\vee/\Lambda]$ is a well-defined homomorphism. Composing with the inverse of the isomorphism of Proposition \ref{prop:residue-field-to-torsion} we obtain (ii) and (iii). Assertion (iv) then follows from Proposition \ref{prop:neutral-iff-zero}.
\end{proof}

\begin{remark}\label{rmk:Springer-fails}
If  $V=(V,b)$ is an $(A,\epsilon)$-bilinear form over $K$, then so is $V(\pi) := (V,\pi b)$. We obtain a map
\[
	W_A^\epsilon(K) \to W_A^\epsilon(k) \times W_A^\epsilon(k) ,\, [V] \mapsto \Big(\partial[V], \partial[V(\pi)] \Big).
\]
If $K$ is complete, $A=\cO_K$ and $\epsilon=1$, then this is precisely the
isomorphism $W(K) \to W(k) \times W(k)$ of Springer \cite{Springer55}. 

In general, however, the map fails to be an isomorphism.  For example, if $G$ is a $p$-group and $K=\bQ_p$, then under the isotypical decomposition of Theorem \ref{thm:isotypical-decomposition} the group $W_G(\bQ_p)$  can have many non-trivial components (corresponding to the irreducible symmetrically  self-dual representations of $G$ over $\bQ_p$), whereas $W_G(\bF_p) = W(\bF_p)$.

For an explicit example, take $p=2$, and $G=\bZ/2\bZ$ acting on the hyperbolic plane $H$ over $\bQ_2$ by interchanging the two isotropic lines. Then $[H]$ is a non-zero element in the kernel of $W_G(\bQ_2) \to W_G(\bF_2)\times W_G(\bF_2)$.
\end{remark}

\begin{remark}\label{rmk:structure}
Let $M$ be a bounded simple $A_K$-module and $\Lambda\subset M$ an $A$-lattice. 
Let $\widebar M_1,\ldots \widebar M_n$ be the (distinct) simple $A_k$-modules that occur as quotient in a Jodan-H\"older filtration of $\Lambda/\pi\Lambda$. Then under the decomposition of Theorem \ref{thm:isotypical-decomposition} the map $\partial$ restricts to a map
\[
	W_A^\epsilon(K,M) \longto \bigoplus_{i=1}^n W_A^\epsilon(k,\widebar M_i).
\]
Moreover, one can use the theory of Morita equivalence of \cite[\S~I.9]{Knus91} to compute the groups $W_A^\epsilon(k,\widebar M_i)$ in terms of the usual Witt groups $W(k)$. Especially when $k$ is finite, this gives quite a bit of control over the map $\partial\colon W_A^\epsilon(K)\to W_A^\epsilon(k)$, and makes it plausible that Theorem \ref{thm:B} will find applications beyond its use in the proof of Theorem \ref{bigthm:even-unimodular-charpol}. 
\end{remark}

\subsection{Sketch of a CAT(0)-proof}Since  it may be of independent interest, we briefly sketch a different, more geometric proof of Theorem \ref{bigthm:reduction}, based on ideas of Goldman-Iwahori \cite{GoldmanIwahori63} and Bruhat-Tits \cite{BruhatTits84,BruhatTits87}. The argument does not seem to generalize to the more general setting of Theorem \ref{thm:B}.

Let $V$ be a finite-dimensional vector space over a discretely valued field $K$. A \emph{valuation} on $V$ is a map $\alpha\colon V \to \bR \cup \{\infty\}$ satisfying
\begin{enumerate}
\item $\alpha(x) = \infty$ if and only if $x=0$,
\item $\alpha(ax) = v(a) + \alpha(x)$ for all $a\in K$, $x\in V$,
\item $\alpha(x+y) \geq \inf \{ \alpha(x), \alpha(y) \}$ for all $x,y\in V$.
\end{enumerate}
If $\alpha$ is a valuation on $V$, then
\[
	\alpha^\vee(\xi) := \inf_{x\in V} \big( v(\xi(x)) - \alpha(x) \big)
\]
defines a valuation on $V^\vee$. If $b\colon V\times V\to K$ is a bilinear form, then we say that $\alpha$ is \emph{reflexive} if $\alpha = \alpha^\vee$ under the identification $V\cong V^\vee$ defined by $b$.

Let $\cB(V,b)$ be the set of all reflexive valuations on $(V,b)$. A \emph{metric} on $\cB(V,b)$ is given by
\[
	 d(\alpha,\beta) := \sup_{x\in V} | \beta(x) - \alpha(x) |.
\]
Using the results of \cite{GoldmanIwahori63}  and \cite[\S~1]{BruhatTits84} one can show that the metric space $\cB(V,b)$ is a CAT(0)-space. In particular, it is complete and uniquely geodesic. The group $\rmO(V,b)$ acts isometrically on $\cB(V,b)$, and if the residue characteristic is different from $2$, then one can identify $\cB(V,b)$ with the spherical Bruhat-Tits building of $\SO(V,b)$, see \cite[Thm.~2.12]{BruhatTits87}.

Let $\cL(V,b)$ be the set of almost unimodular lattices in $V$. We have $\rmO(V,b)$-equivariant maps
\[
	\cB(V,b) \to \cL(V,b), \, \alpha \mapsto \Lambda_\alpha := \{ x \in V \mid \alpha(x) \geq 0 \}
\]
and
\[
	\cL(V,b) \to \cB(V,b),\, \Lambda \mapsto \alpha_\Lambda := \big( x \mapsto \inf \{ v(\pi^n) \mid \pi^n x \in \Lambda \} \big).
\]
The first map is \emph{semi-continuous} is the following sense: every $\alpha \in \cB(V,b)$ has an open neighbourhood $U$ so that for all $\beta \in U$ we have $\Lambda_\beta  \subset \Lambda_\alpha$. The first map is a section of the second: for all almost unimodular lattices $\Lambda$ we have $\Lambda_{\alpha_\Lambda} = \Lambda$.

\begin{proof}[Sketch of alternative proof of Theorem \ref{bigthm:reduction}] We only give the proof of the hardest part: the independence of $\Lambda$ of $[\Lambda^\vee/\Lambda]\in W_G(k)$. Let $(V,b,\rho)$ be a $G$-bilinear form over $K$, and let $\Lambda_0$ and $\Lambda_1$ be almost unimodular $G$-lattices in $V$. If $\Lambda_0$ and  $\Lambda_1$ are comparable ($\Lambda_0\subset\Lambda_1$ or $\Lambda_1\subset \Lambda_0$), then it is easy to see that $[\Lambda_0^\vee/\Lambda_0]=[\Lambda_1^\vee/\Lambda_1]$ in $W_G(k)$. 

Otherwise, since $\cB(V,b)$ is CAT(0), there exists a unique geodesic
\[
	\alpha \colon [0,1] \to \cB(V,b),\, t \mapsto \alpha_t
\]
such that $\alpha_0 = \alpha_{\Lambda_0}$ and $\alpha_1= \alpha_{\Lambda_1}$. By unicity, we have that $\alpha_t$ is fixed by $G$ for all $t$, and hence that  $\Lambda_{\alpha_t}$ is an almost unimodular $G$-lattice for all $t$. By semi-continuity (and compactness of [0,1]), we find a finite sequence of almost unimodular $G$-lattices $(\Lambda_{t_i})_i$ with $\Lambda_{t_0} = \Lambda_0$, $\Lambda_{t_1} = \Lambda_1$, and with $\Lambda_{t_i}$ and $\Lambda_{t_{i+1}}$ comparable for all $i$.
\end{proof}

In similar spirit, the existence of almost unimodular $G$-lattices (Corollary \ref{cor:almost-unimodular-lattices-exist}) can be deduced from the Bruhat-Tits fixed point theorem on $\cB(V,b)$.

\section{One-dimensional hermitian forms and the twisting group \texorpdfstring{$\mu(E,\sigma)$}{\textbackslash mu(E,\textbackslash sigma)}}

The remainder of this paper is about $\bZ$-bilinear forms, that is, symmetric bilinear forms $(V,b)$ equipped with an action $\rho\colon \bZ \to \rmO(V,b)$ of the infinite cyclic group $\bZ$.

We beginning by recalling the  `hermitian' construction of $\bZ$-bilinear forms, and describe the local-global obstruction that will play a crucial role in the proof of Theorem \ref{bigthm:even-unimodular-charpol}. All material in this section is well-known.

\subsection{Hermitian construction of $\bZ$-bilinear forms}\label{sec:hermitian-construction}

Let $K$ be a field of characteristic different from $2$, let $E_0$ be an \'etale $K$-algebra, and let $E$ be an \'etale $E_0$-algebra that is free of rank $2$ over $E_0$. Let $\sigma$ denote the canonical involution of $E$, fixing $E_0$. Every $\lambda \in E_0^\times$ defines a symmetric bilinear form
\[
	b_\lambda\colon E\times E \to K,\, (x,y) \mapsto \tr_{E/K} ( \lambda x \sigma(y) )
\]
If $\alpha \in E^\times$ satisfies $\alpha \sigma(\alpha)=1$, then multiplication by $\alpha$ is an isometry of $(E,b_\lambda)$, and the homomorphism
\[
	\rho_\alpha\colon \bZ\to \rmO(E,b_\lambda),\, 1 \mapsto \alpha
\]
makes $(E,b_\lambda,\rho_\alpha)$ into a $\bZ$-bilinear form.

\subsection{The twisting group}\label{subsec:twisting-group}
Consider the group
\[
	\mu(E,\sigma) := \frac{E_0^\times}{ \{z\sigma(z) \mid z \in E^\times \}}.
\]
Up to $E$-linear isometry, the pair $(E,b_\lambda)$ only depends on the class of $\lambda$ in $\mu(E,\sigma)$. Note that $\mu(E,\sigma)$ is trivial if $E=E_0\times E_0$.

\begin{lemma} Let $T$ be the group scheme over $K$ defined by the short exact sequence
\[
	1 \longto T \longto \Res_{E/K} \bG_{m,E} \overset{\Nm}{\longto} \Res_{E_0/K} \bG_{m,E_0} \longto 1.
\] 
Then $\mu(E,\sigma) = \rH^1(K,T)$.
\end{lemma}

\begin{proof}
This follows from the long exact sequence of cohomology together with Hilbert 90.
\end{proof}

\begin{remark}
The group scheme $T$ acts by isometries on the `standard' form
\[
	b_1\colon E\times E \to K,\, (x,y) \mapsto \tr_{E/K} (  x \sigma(y) ),
\]
and $b_\lambda$ is the twist of $b_1$ by the class in $\rH^1(K,T) = \mu(E,\sigma)$ determined by $\lambda$.
\end{remark}

We can express $\mu(E,\sigma)$ in terms of Brauer groups of $E$ and $E_0$:

\begin{lemma}\label{lemma:trivial-mu}\label{lemma:mu-to-brauer}
There is an exact sequence
\[
	1 \longto \mu(E,\sigma) \overset\beta\longto \Br E_0 \longto \Br E,
\]
where the map $\Br E_0 \to \Br E$ is the base change map.
\end{lemma}

\begin{proof}
This follows from Hilbert 90 and the long exact sequence of cohomology induced by the short exact sequence
\[
	1 \longto  \Res_{E_0/K} \bG_{m,E_0} \longto \Res_{E/K} \bG_{m,E} \longto T \longto 1
\]
of group schemes over $K$, where the last map sends $z$ to $\sigma(z)/z$.
\end{proof}

\subsection{Hasse-Witt invariants of a hermitian form}
The Hasse-Witt invariant $\epsilon(b_\lambda) \in \Br K$ is determined by the Hasse-Witt invariant $\epsilon(b)$ and by the twisting cocycle $\lambda \in \mu(E,\sigma)$, as follows.

\begin{proposition}[{\cite[Thm.~4.3]{BCM04}}]\label{prop:HW-of-twisted-form}
For every $\lambda \in E_0^\times$ we have
\[
	\epsilon(b_\lambda) = \epsilon(b_1) + \Nm_{E_0/K} \beta(\lambda)
\]
in $\Br K$.\qed
\end{proposition}

\subsection{Twisting group of local and global fields} If $E$ is a local or global field, then the group $\mu(E,\sigma)$ can easily be made explicit using the standard descriptions of Brauer groups of local and global fields.

\begin{lemma}Assume that $E$ is a local field. Then we have a natural commutative diagram
\[
\begin{tikzcd}
1 \arrow{r} & \mu(E,\sigma)  \arrow{r} \arrow{d}{\theta} & \Br E_0 \arrow{r} \arrow{d}{\inv} & \Br E \arrow{r} \arrow{d}{\inv} & 1 \\
0 \arrow{r} & \bZ/2\bZ \arrow{r}{1/2} & \bQ/\bZ \arrow{r}{2} & \bQ/\bZ \arrow{r} & 0
\end{tikzcd}
\]
in which the vertical map $\theta\colon\mu(E,\sigma) \to \bZ/2\bZ$ is an isomorphism.
\end{lemma}

\begin{proof}
If $E$ is non-archimedean then the two maps $\inv$ are isomorphisms, and the lemma follows from Lemma \ref{lemma:mu-to-brauer}. If $E$ is archimedean, then necessarily $E=\bC$ and $E_0=\bR$, and again the lemma follows from Lemma \ref{lemma:mu-to-brauer}.
\end{proof}

\begin{remark}
We will use explicit descriptions of $\theta$ in the following cases:
\begin{enumerate}
\item if $(E,E_0) = (\bC,\bR)$ then $\theta(\lambda) = \sgn(\lambda)$,
\item if $E/E_0$ is an unramified extension of non-archimedean local fields, then $\theta(\lambda)=v_{E_0}(\lambda) \bmod 2$,
\item if $E/E_0$ is a ramified extension of non-archimedean local fields of odd residue characteristic, and if $\pi_E$ is a uniformizer of $E$ then $\pi_{E_0} := \Nm \pi_E$ is a uniformizer of $E_0$ and $\theta$  is given by mapping $\lambda\pi_{E_0}^n$  with $\lambda \in \cO_{E_0}^\times$ to the class of $\lambda$ in $\ell^\times/(\ell^\times)^2 \cong \bZ/2\bZ$, where $\ell$ is the residue field of $\cO_{E_0}$.
\end{enumerate}
\end{remark}

\begin{theorem}[Local-global obstruction]\label{thm:local-global}Assume that $E$ is a global field. Let $\cS$ be the set
of places $w$ of $E_0$ for which
\[
	E_w := E \otimes_{E_0} E_{0,w}
\]
is a field. Then the sequence
\[
	1 \longto \mu(E,\sigma) \longto \bigoplus_{w\in\cS} \mu(E_w,\sigma) \overset{\sum \theta_w}\longto \bZ/2\bZ \longto 0
\]
is exact.
\end{theorem}

\begin{proof}
Class field theory gives a short exact sequence
\[
	1 \longto \Br E_0 \longto \bigoplus\limits_w \Br E_{0,w} \overset{\sum \inv}{\longto} \bQ/\bZ \longto 0,
\]
and a similar sequence for $E$. These two sequences sit in a commutative diagram
\[
\begin{tikzcd}[column sep=1.7em]
1 \arrow{r} & \Br E_0 \arrow{d} \arrow{r} 
	& \Big( \bigoplus\limits_{w\in \cS} \Br E_{0,w} \Big) \oplus  \Big( \bigoplus\limits_{w\not\in \cS}  \Br E_{0,w} \Big) \arrow{d} \arrow{r} 
	& \bQ/\bZ \arrow{d}{2} \arrow{r} & 1 \\
1 \arrow{r} & \Br E \arrow{r} 
	& \Big( \bigoplus\limits_{w\in \cS} \Br E_{w} \Big) \oplus  \Big( \bigoplus\limits_{w\not\in \cS}  \Br E_{0,w} \times \Br E_{0,w} \Big) \arrow{r} 
	& \bQ/\bZ \arrow{r} & 1 
\end{tikzcd}
\]
The kernels of the vertical maps give an exact sequence
\[
	1 \longto \mu(E,\sigma) \longto \bigoplus_{w\in\cS} \mu(E_w,\sigma) \overset{\sum\theta_w}\longto \bZ/2\bZ 
\]
and the rightmost map is surjective since $\cS$ is non-empty.
\end{proof}

\section{Lattices in one-dimensional hermitian forms}

 Let $K$ be a discrete valuation field with residue field $k$, maximal ideal $\fm_K$ and uniformizing element $\pi_K$.  Let $E$ be a finite separable field extension of $K$, and let $\sigma$ be a non-trivial involution of $E$ over $K$. Denote by $E_0$ the fixed field of $\sigma$, and by $\ell$ the residue field of $E$. Let $\lambda \in E_0^\times$, and consider the associated $K$-bilinear form $b_\lambda$ on $E$ defined in \S~\ref{sec:hermitian-construction}. In this section we study lattices in $(E,b_\lambda)$, and the image of a $\bZ$-bilinear form $(E,b_\lambda,\rho_\alpha)$ under $W_\bZ(K) \to W_\bZ(k)$.

\subsection{Almost unimodular \texorpdfstring{$\cO_E$}{O\_E}-lattices}
We construct an $\cO_E$-lattice $\Lambda$ in $(E,b_\lambda)$ satisfying
$\pi_E \Lambda^\vee \subset \Lambda \subset \Lambda^\vee$ and explicitly determine the $k$-bilinear form
\begin{equation}\label{eq:reduced-form}
	\frac{\Lambda^\vee}{\Lambda} \times \frac{\Lambda^\vee}{\Lambda} \overset{b}{\longto} 
	\frac{\pi_K^{-1}\cO_K}{\cO_K} \overset{\pi_K}\longto k
\end{equation}
on the $\ell$-module $\Lambda^\vee/\Lambda$.

Denote the valuation of the different ideal $\cD_{E/K}$ by $\delta$, so that $\cD_{E/K}=\fm_E^\delta$. 

\begin{lemma}\label{lemma:computation-trace}\label{lemma:computation-dual}
For every $n$, the dual of $\fm_E^n$ with respect to $b_\lambda$ is $ \fm_E^{-n-\delta-v_E(\lambda)}$.
\end{lemma}

\begin{proof}
The dual is an $\cO_E$-module and hence it equals $\fm_E^m$ where $m$ is the smallest integer such that $b_\lambda(\fm_E^n,\fm_E^m) \subset \cO_E$. By \cite[\S~III.3, Prop.~7]{Serre62} we have $b_\lambda(\fm_E^n,\fm_E^m) \subset \cO_E$ if and only if $\lambda \fm_E^{n+m} \subset \fm_E^{-\delta}$, hence $m=-\delta-n-v_E(\lambda)$.
\end{proof}

\begin{corollary} \label{cor:even-case}
 If $v_E(\lambda) + \delta = -2n$, then the $\cO_K$-lattice $\Lambda := \fm_E^n$ in the $K$-bilinear form $(E,b_\lambda)$ satisfies $\Lambda^\vee = \Lambda$.\qed
\end{corollary}

Denote the maximal unramified extension of $K$ in $E$ by $L$. We have $\cO_L/\fm_L=\ell$.

\begin{proposition}\label{prop:odd-case}
If $v_E(\lambda) +  \delta = 1-2n$,  then 
\begin{enumerate}
\item the $\cO_K$-lattice $\Lambda := \fm_E^n$ satisfies  $\pi_E \Lambda^\vee = \Lambda$
\item the element
\[
	u :=  \tr_{E/L}  \big( \lambda \pi_K   \pi_E^{n-1} \sigma(\pi_E^{n-1})   \big)
\]
is a $\sigma$-invariant unit in $\cO_L$
\item the induced $k$-bilinear form on the one-dimensional $\ell$-vector space $\Lambda^\vee/\Lambda$  is isomorphic with the form
\[
	\ell\times \ell \to k,\, (x,y) \mapsto 
	\tr_{\ell/k} \big( \bar{u}  \cdot x \sigma(y)  \big).
\]
\end{enumerate}
\end{proposition}

%

\begin{proof}
The first assertion follows from Lemma \ref{lemma:computation-dual}. For the second, note that $u$ is $\sigma$-invariant by construction. Observe that
\[
	\lambda \pi_K   \pi_E^{n-1} \sigma(\pi_E^{n-1}) \in \lambda \pi_K \fm_E^{2n-2} \subset \lambda \fm_E^{2n-1} = \fm_E^{-\delta}
\]
so that $u  \in \cO_L$ by \cite[\S~III.3, Prop.~7]{Serre62}. That $u$ is a unit, will follow from the third assertion, which we now prove. 

Let $\xi,\eta$ be elements of $\Lambda^\vee/\Lambda$. We will compute their pairing in $k$ under
\[
	\frac{\Lambda^\vee}{\Lambda} \times \frac{\Lambda^\vee}{\Lambda} \longto
	\frac{ \pi^{-1}\cO_K}{\cO_K} \overset{\pi_K}\longto k.
\]
We have $\Lambda^\vee/\Lambda = \fm_E^{n-1}/\fm_E^{n} \cong \ell$ and we can find $x,y\in \cO_L$
with $\xi = x\pi_E^{n-1}$ and $\eta = y\pi_E^{n-1}$. 
Using the linearity and transitivity properties of the trace, we find
\begin{eqnarray*}
	\pi_K b_\lambda(x\pi_E^{n-1}, y\pi_E^{n-1}) 
	&=&\tr_{E/K} \big( \pi_K \lambda x \pi_E^{n-1} \sigma(y\pi_E^{n-1}) \big) \\
	&=& \tr_{L/K} \tr_{E/L} \big(  \lambda \pi_K \pi_E^{n-1} \sigma(\pi_E^{n-1})  x\sigma(y) \big) \\
	&=&  \tr_{L/K} \big( u x \sigma(y)  \big).
\end{eqnarray*}
It follows that $\tr_{L/K} (u x\sigma(y)) \in \cO_K$ and that $\xi$ and $\eta$ pair to
$\tr_{L/K}(u x\sigma(y)) \bmod \fm_K$ in $k$.  

Since 
the pairing on $\Lambda^\vee/\Lambda$ is perfect, and since $L/K$ is unramified, we conclude that
$u$ must be a unit in $\cO_L$ and that the induced pairing on $\Lambda^\vee/\Lambda$ is as described in the third assertion. 
\end{proof}

\subsection{Image of a hermitian form under the map
\texorpdfstring{$W_\bZ^b(K) \to W_\bZ(k)$}{W\_Z(K) -> W\_Z(k)}}

  Fix an element  $\alpha \in E^\times$ with $\alpha\sigma(\alpha)=1$ and $\alpha\neq \pm1$. For every $\lambda$ in $E_0^\times$ we have
a $\bZ$-bilinear form $(E,b_\lambda,\rho_\alpha)$ over $K$. Since $\alpha$ is a unit in $\cO_{E}$, this $\bZ$-bilinear form is bounded. Every sub-$\cO_E$-module $\Lambda \subset E$ is stable under the action of $\rho_\alpha(\bZ)=\alpha^\bZ$, so we can use Corollary \ref{cor:even-case} and Proposition \ref{prop:odd-case} to say something about the image of $[E,b_\lambda,\rho_\alpha]$ under the map $\partial \colon W^b_\bZ(K) \to W_\bZ(k)$ of Theorem \ref{bigthm:reduction}. 

\begin{proposition}[Unramified case]\label{prop:unramified-case}
If $E/E_0$ is unramified, then there exists a $\lambda\in \mu(E,\sigma)$ such that
$\partial[E,b_\lambda,\rho_\alpha]=0$ in $W_\bZ(k)$. 
\end{proposition}

\begin{proof}
Since $E/E_0$ is unramified, there exists a $\lambda \in E_0^\times$ such that $v_E(\lambda) + \delta$ is even, and by Corollary \ref{cor:even-case} there is an $\cO_E$-module $\Lambda \subset E$ with $\Lambda^\vee = \Lambda$. This module is stable under $\alpha$, hence $\partial [E,b_\lambda,\rho_\alpha] = [\Lambda^\vee/\Lambda]=0$.
\end{proof}

\begin{lemma}
If $E/E_0$ is ramified, then $\bar\alpha := \alpha \bmod \fm_E$ satisfies $\bar\alpha \in \{\pm 1\}$.
\end{lemma}

\begin{proof} We have $\sigma_\ell = \id$ and $\bar\alpha \sigma_\ell(\bar\alpha)=1$.
\end{proof}

\begin{proposition}[Ramified with odd residue characteristic]\label{prop:ramified-case}
Assume that $E/E_0$ is ramified, and that the residue characteristic is odd. Let $\chi\colon \bZ \to k^\times$ be the character that maps $1$ to $\bar \alpha$. Then for each of the two classes  $\gamma \in W(k)$ with $\dim \gamma \equiv [\ell : k] \bmod 2$ there is a unique
$\lambda \in \mu(E,\sigma)$ such that
\[
	\partial [ E, b_\lambda, \rho_\alpha ] = \gamma
\]
in $W(k)=W_\bZ(k,\chi)\subset W_\bZ(k)$.
\end{proposition}

Here $W_\bZ(k,\chi)$ denotes the subgroup of $W_\bZ(k)$ generated by the $\bZ$-bilinear forms of the form $(k,b,\chi)$. It is one of the components in the decomposition of Theorem \ref{thm:isotypical-decomposition}.

\begin{proof}[Proof of Proposition \ref{prop:ramified-case}]
Since the residue characteristic is odd, the quadratic extension $E/E_0$ is tamely ramified and $\cD_{E/E_0} = \fm_E$. The transitivity formula $\cD_{E/K} = \cD_{E_0/K} \cD_{E/E_0} $ implies
\[
	v_E(\lambda)  + \delta  = 2 v_{E_0} (\lambda) + 2v_{E_0}(\cD_{E_0/K}) + 1
\]
and hence $v_E(\lambda)  + \delta$ is odd and Proposition \ref{prop:odd-case} applies. Since $\sigma_\ell =\id$, we see that $\partial [E,b_\lambda,\rho_\alpha]$ can be represented by the bilinear form
\[
	\bar{b}\colon \ell \times \ell \to k,\, (x,y) \mapsto \tr_{\ell/k} (\bar{u} xy)
\]
on which $\bZ$ acts via $\chi$. The group $W(k)$ has four elements, distinguished by their dimension in $\bZ/2\bZ$ and determinant in $k^\times /(k^\times)^2$. For the form $\bar{b}$ we have
\begin{eqnarray*}
	\dim \bar{b} &=& [\ell:k] \in \bZ/2\bZ, \\
	\det \bar{b} &=&  - (-1)^{[\ell:k]} \Nm (\bar u) \in k^\times/(k^\times)^2.
\end{eqnarray*}
Since the norm map $\ell^\times \to k^\times$ is surjective, we see that by changing $\lambda$ by a unit in $\cO_{E_0}^\times$, we can reach both classes in $W(k)$ of dimension $[\ell:k]$.
\end{proof}

\begin{proposition}[Ramified with even residue characteristic]\label{prop:even-case}
Assume that $E/E_0$ is ramified, and that the residue characteristic is even. Then for all
$\lambda \in \mu(E,\sigma)$ we have 
\[
	\partial [ E, b_\lambda, \rho_\alpha ] = [\ell:k] \delta 
\]
in $W_\bZ(k,{\mathbf{1}}) = W(k) = \bZ/2\bZ$.
\end{proposition}

\begin{proof}[Proof of Proposition \ref{prop:even-case}]
Note that $v_E(\lambda)$ is even. If $\delta$ is even, then Corollary \ref{cor:even-case} implies that $\partial [E,b_\lambda, \rho_\alpha ]=0$. If $\delta$ is odd, then Proposition \ref{prop:odd-case} shows that 
$\partial [E,b_\lambda,\rho_\alpha]$ can be represented by a bilinear form of dimension $[\ell:k]$.
\end{proof}

\subsection{Relation to the characteristic polynomial} 

\begin{lemma}\label{lemma:S-at-one}
Assume that $E/E_0$ is ramified. Let $S$ be the characteristic polynomial of $\alpha$ over $K$. Then
 $ v_K(S(1))+ v_K(S(-1)) \equiv [\ell:k] \delta \bmod 2$. If moreover $k$ has odd characteristic, then
 either 
\begin{enumerate} 
\item$\bar\alpha = 1$ then $ v_K(S(1)) \equiv [\ell:k] \bmod 2$, or
\item  $\bar\alpha = -1$ then $ v_K(S(-1)) \equiv [\ell:k] \bmod 2$.
\end{enumerate}
\end{lemma}

\begin{proof}
The transitivity formula for the different shows $\delta \equiv v_E(\cD_{E/E_0}) \bmod 2$, and hence
$\delta \equiv v_E(\sigma(\alpha)-\alpha) \bmod 2$.  Since $\alpha$ and $\sigma(\alpha)$ are mutually inverse units, we have
 \[
 	v_E(1-\alpha) + v_E(1+\alpha) = v_E(1-\alpha^2) = v_E(\sigma(\alpha)-\alpha) \equiv \delta \bmod 2
\]
Since all the roots of $S$ have the same valuation, we find
\[
	v_K(S(\pm 1)) = [E:K] \frac{v_E(1 \mp \alpha) }{ e(E:K) } = [ \ell : k ] v_E(1 \mp \alpha),
\]
where $e(E:K)$ denotes the ramification index.
Comparing the above expressions yields $v_K(S(1))+v_K(S(-1)) \equiv [\ell:k]\delta \bmod 2$. 

Now if the residue characteristic is odd then $\delta \equiv 1 \bmod 2$. If moreover $\bar\alpha=1$, then $1+\alpha$ is a unit and $v_K(S(-1))=0$,  and similarly, if  $\bar\alpha=-1$, then $1-\alpha$ is a unit and $v_K(S(1))=0$.
\end{proof}

\section{Unimodular lattices in $\bZ$-bilinear forms over local fields}\label{sec:unimodular-in-local}

Let $K$ be a non-archimedean local field, $E_0$ an \'etale $K$-algebra, and $E$ an \'etale $E_0$-algebra which is free of rank $2$ over $E_0$. Denote by $\sigma$ the involution of $E$ fixing $E_0$. Let $\alpha \in E^\times$ satisfy $\alpha \sigma(\alpha)=1$ and $\alpha \neq \sigma(\alpha)$, and denote by $S$ be the characteristic polynomial of $\alpha$ over $K$.

\begin{proposition}\label{prop:at-p}
Assume that the characteristic of $k$ is odd, that $S(1)$ and $S(-1)$ are non-zero, and that $v_K(S(1))$ and $v_K(S(-1))$ are even. Then
there exists a $\lambda \in \mu(E,\sigma)$ and a unimodular $\cO_K$-lattice $\Lambda \subset (E,b_\lambda)$ stable under $\alpha$.
\end{proposition}

\begin{proof}
The algebra $E_0$ decomposes as a product of fields $E_0 = \prod_{w\in \cS} E_{0,w}$, indexed by a finite set $\cS$. For every $w\in \cS$, the algebra $E_w := E \otimes_{E_0} E_{0,w} $ is a quadratic $E_{0,w}$-algebra of exactly one of the following types:
\begin{enumerate}
\item[(sp)] $E_w=E_{0,w}\times E_{0,w}$ 
\item[(un)] $E_w$ is an unramified quadratic extension of $E_{0,w}$
\item[(+)] $E_w$ is a ramified quadratic extension of $E_{0,w}$, and the image $\bar\alpha_v$ of $\alpha$ in the residue field $\ell_w$ of $E_w$ is $1$
\item[(-)]  $E_w$ is a ramified quadratic extension of $E_{0,w}$, and the image $\bar\alpha_v$ of $\alpha$ in the residue field $\ell_w$ of $E_w$ is $-1$
\end{enumerate}
This gives a partition $\cS = \cS_{\sp} \cup \cS_{\un} \cup \cS_{+} \cup \cS_{-}$.

Now choose $\lambda = (\lambda_w)_w$ in $E_0^\times = \prod_w E_{0,w}^\times$ such that 
\begin{enumerate}
\item for every $w \in \cS_{\un}$ we have $\partial[ E_w, b_{\lambda_w}, \alpha ] =0$ 
\item $\sum_{w \in \cS_+} \partial [ E_w, b_{\lambda_w}, \alpha ] = 0$ in $W(k) = W(k,\chi_+) \subset W_\bZ(k)$
\item $\sum_{w \in \cS_-} \partial [ E_w, b_{\lambda_w}, \alpha ] = 0$ in $W(k) = W(k,\chi_-) \subset W_\bZ(k)$
\end{enumerate}
where $\chi_{\pm 1}$ denotes the character $\bZ \to k^\times, 1 \mapsto \pm 1$. Such $\lambda=(\lambda_w)$ indeed exists. For (i) this follows from Propostion \ref{prop:unramified-case}. For (ii) and (iii) note that by Lemma \ref{lemma:S-at-one} and the condition on $S(\pm 1)$ we have
\[
	\sum_{w\in \cS_+} [\ell_w:k] \equiv \sum_{w\in \cS_-} [\ell_w:k] \equiv 0 \pmod 2.
\]
Hence it follows from Proposition \ref{prop:ramified-case} that we can choose $(\lambda_w)_{w\in \cS_{\pm}}$ as required.

Since $(E_w, b_{\lambda_w}, \alpha)$ is neutral for all $w\in \cS_{\sp}$, we conclude that 
\[
	\partial [ E,b_\lambda,\rho_\alpha ] = \sum_{w \in \cS} \partial [ E_w, b_{\lambda_w}, \rho_{\alpha_w} ] = 0
\]
in $W_\bZ(k)$, and Theorem \ref{bigthm:reduction} gives us that $(E,b_\lambda,\rho_\alpha)$ contains a unimodular $\bZ$-stable lattice.
\end{proof}

\begin{proposition}\label{prop:local-unimodular-existence-at-2}
Assume that the characteristic of $k$ is even, that $S(1)$ and $S(-1)$ are non-zero, and that $v_K(S(1))+v_K(S(-1))$ is even. Then
there exists a $\lambda \in \mu(E,\sigma)$ and a unimodular $\cO_K$-lattice $\Lambda \subset (E,b_\lambda)$ stable under $\alpha$. 
\end{proposition}

\begin{proof}
The algebra $E_0$ decomposes as a product of fields $E_0 = \prod_{w\in \cS} E_{0,w}$. For every $w\in \cS$, the algebra $E_w := E \otimes_{E_0} E_{0,w} $ is a quadratic $E_{0,w}$-algebra of exactly one of the following types:
\begin{enumerate}
\item[(sp)] $E_w=E_{0,w}\times E_{0,w}$ 
\item[(un)] $E_w$ is an unramified quadratic extension of $E_{0,w}$
\item[($\pm$)] $E_w$ is a ramified quadratic extension of $E_{0,w}$
\end{enumerate}
This gives a partition $\cS = \cS_{\sp} \cup \cS_{\un} \cup \cS_{\pm}$. 

Now choose $\lambda= (\lambda_w)_w$ in $E_0^\times = \prod_w E_{0,w}^\times$ such that 
for every $w \in \cS_{\un}$ we have $\partial[ E_w, b_{\lambda_w}, \alpha ] =0$.
We have $\partial [E_w, b_{\lambda_w}, \rho_{\alpha_w} ]=0$ for $w\in \cS_{sp}$ and using proposition \ref{prop:even-case}  we find
\[
	\partial [ E,b_\lambda,\rho_\alpha ] = \sum_{w \in \cS_{\pm}} \partial [ E_w, b_{\lambda_w}, \rho_{\alpha_w} ]
	= \sum_{w\in \cS_{\pm}} [\ell_w : k ] \delta_w
\]
in $\bZ/2\bZ = W(k) \subset W_\bZ(k)$. Lemma \ref{lemma:S-at-one} then shows that $\partial [E,b_\lambda,\rho_\alpha ] = 0$, and hence by
Theorem \ref{bigthm:reduction} we conclude that $(E,b_\lambda,\rho_\alpha)$ contains a unimodular $\bZ$-stable lattice.
\end{proof}

\begin{remark}\label{rmk:freedom-of-lambda}
Note that the components $\lambda_w$ in the above proof can be chosen arbitrarily for $w \in \cS_{\sp}$ and $w \in \cS_{\pm}$. The only restriction concerns the places $w\in \cS_{un}$.
\end{remark}

\section{An intermezzo on $2$-adic lattices}

By itself, Theorem \ref{bigthm:reduction} does not say anything about the existence of a $G$-stable \emph{even} unimodular lattice in a $G$-bilinear form over $\bQ_2$. The aim of this section is to establish the following criterion, which will be used in the proof of Theorem \ref{bigthm:even-unimodular-charpol}.

\begin{theorem}\label{thm:evenness}
Let $G$ be a group and $(V,b,\rho)$ be a $G$-bilinear form over $\bQ_2$. Assume that $\rho(G) \subset \SO(V,b)$. Then $V$ contains a $G$-stable even unimodular lattice if and only if the following three conditions hold
\begin{enumerate}
\item $(V,b,\rho)$ contains a $G$-stable unimodular lattice,
\item $(V,b)$ contains an even unimodular lattice,
\item for every $g\in G$ we have $v_2 (\delta(\rho(g))) = 0$ in $\bZ/2\bZ$.
\end{enumerate}
\end{theorem}

Here $\delta\colon \SO(V,b) \to \bQ_2^\times/ (\bQ_2^\times)^2$ denotes the spinor norm. We will recall its definition in \S~\ref{subsec:spinor-norm}. 

\subsection{Classification}
We start by recalling some results on the classification of unimodular lattices over $\bZ_2$. 

\begin{proposition}\label{prop:2-adic-unicity}
Let $(V,b)$ be a bilinear form over $\bQ_2$ and let $\Lambda_1$ and $\Lambda_2$ be unimodular lattices in $V$. If $\Lambda_1$ and $\Lambda_2$ are either both even, or both odd, then there exists a $g\in \rmO(V,b)$ such that $g\Lambda_1=\Lambda_2$. 
\end{proposition}

\begin{proof}Follows from \cite[Thm.~93.29]{OMeara73}.
\end{proof}

Denote by $H$ and $N$ the module $\bZ_2\oplus \bZ_2$ equipped with the bilinear forms given in terms of Gram matrices by
\[
	b_H :=  \left(\begin{array}{cc}0 & 1 \\1 & 0\end{array}\right), \quad
	b_N := \left(\begin{array}{cc}2 & -1 \\-1 & 2\end{array}\right).
\]

\begin{proposition}\label{prop:2-adic-classification}
Every even unimodular $\bZ_2$-lattice is either of the form
$H^n$ or of the form $N \oplus H^{n-1}$.
\end{proposition}

\begin{proof}See \cite[Prop.~5.2]{BelolipetskyGan05} or \cite[Satz~15.6]{Kneser02}
\end{proof}

We have $\disc H^n=1$ and $\disc (N \oplus H^{n-1})=-3$ in $\bZ_2^\times/(\bZ_2^\times)^2$.


\begin{corollary}\label{cor:suitable-HW}
For every $n>0$ and $d \in \{1,-3\} \subset \bQ_2^\times/(\bQ_2^\times)^2$ there is a unique $\epsilon \in \{\pm 1\} = (\Br \bQ_2)[2]$ such that the unique bilinear form $V$ over $\bQ_2$ of dimension $2n$, discriminant $d$ and Hasse-Witt invariant $\epsilon$ contains an even unimodular lattice.\qed
\end{corollary}


\subsection{The spinor norm}\label{subsec:spinor-norm}

If $(V,b)$ is a bilinear form over a field $K$ of characteristic different from $2$, then the short exact sequence
\[
	1 \longto \{\pm1\} \longto \Spin(V,b) \longto \SO(V,b) \longto 1
\]
induces a morphism
\[
	\delta\colon \SO(V,b) \to \rH^1(K,\{\pm 1\}) = K^\times/(K^\times)^2
\]
called the \emph{spinor norm}.

\begin{theorem}[Zassenhaus formula]\label{thm:zassenhaus}
Let $K$ be a field of characteristic different from $2$. Let $(V,b)$ be a bilinear form of dimension $n$ over $K$, let $\alpha \in \SO(V)$. Let $V_0\subset V$ be the maximal subspace on which $1+\alpha$ is nilpotent, and $V_1$ its orthogonal complement. Then
\[
	\delta(\alpha) = \det (b_{|V_0} ) \cdot \det( \tfrac{1+\alpha}{2}, V_1 )
\]
in $ K^\times /(K^\times)^2$.
\end{theorem}

\begin{proof}See \cite{Zassenhaus62} or \cite[Thm.~C.5.7]{Conrad14}.
\end{proof}

\begin{proposition}\label{prop:spinor-norm-on-even-unimodular}
Let $(V,b)$ be a bilinear form over $\bQ_2$. Assume that $\alpha\in \SO(V,b)$ stabilizes an even unimodular $\bZ_2$-lattice $\Lambda \subset V$. Then $v_2(\delta(\alpha))\equiv 0 \bmod 2$.
\end{proposition}

\begin{proof}
Since $\Lambda$ is even and unimodular, the quadratic form $q(x):=b(x,x)/2$ on the $\bZ_2$-module $\Lambda$ is non-degenerate (or regular), and the short exact sequence of the spin cover (over $\bQ_2$) extends to a short exact sequence
\[
	1 \longto \mu_{2}  \longto \Spin(\Lambda) \longto \SO(\Lambda) \longto 1
\]
of group schemes over $\Spec \bZ_2$, exact in the fppf topology, see \cite[\S~C.4]{Conrad14}. It follows that $\delta(\alpha) \in \bZ_2^\times/(\bZ_2^\times)^2$ for all $\alpha \in \SO(\Lambda)$, as we had to show. See \cite[Lemma~4.3]{GuntherNebe09} for an alternative proof.
\end{proof}

\subsection{Existence of invariant even unimodular lattices}
We are now ready to prove our criterion.

\begin{proof}[Proof of Theorem \ref{thm:evenness}]
It is clear that (i) and (ii) are necessary, and Proposition \ref{prop:spinor-norm-on-even-unimodular} shows that also (iii) is necessary. So we are left with showing that (i)--(iii) imply the existence of a $G$-stable even unimodular lattice.

This follows from a more or less straightforward calculation, based on \cite[\S~5]{BelolipetskyGan05}. We give the argument in case $\disc b =1$. The case $\disc b = -3$ is similar (and will not be used in the proof of Theorem \ref{bigthm:even-unimodular-charpol}).

Let $\Lambda_0$ be a $G$-stable unimodular lattice. If $\Lambda_0$ is even, then we are done. If not, then by Propositions \ref{prop:2-adic-unicity} and \ref{prop:2-adic-classification},  
we may assume without loss of generality that 
\[
	\Lambda_0 =H^{n-1} \oplus U,
\]
where $U$ is the rank $2$ module $\bZ_2\oplus \bZ_2$ equipped with the odd unimodular form 
\[
	b_U :=  \left(\begin{array}{cc}1 & 0 \\0 & -1\end{array}\right).
\]
If $(e,f)$ denotes the standard basis of $U$, then $\bQ_2\otimes \Lambda_0$ contains 
the even unimodular lattices
\[
	\Lambda_1 = H^{n-1} \oplus \langle e+f, \frac{e-f}{2} \rangle,
\]
and
\[
	\Lambda_2 = H^{n-1} \oplus \langle e-f, \frac{e+f}{2} \rangle,
\]
of discriminant $1$. 

Note that the lattice 
\[
	\Lambda := \{ x \in \Lambda_0 \mid b(x,x) \in 2\bZ_2 \} = H^{n-1} \oplus \langle e+f, e-f \rangle
\]
is preserved by $\SO(\Lambda_0)$. 
There are precisely three unimodular lattices containing $\Lambda$: the odd lattice $\Lambda_0$, and the even lattices $\Lambda_1$ and $\Lambda_2$. It follows that the group 
$\SO(\Lambda_0)$ acts on the set $\{\Lambda_1,\Lambda_2\}$. 

By \cite[Lemma~4.4]{GuntherNebe09} the action on this two-element set is given by the homomorphism
\[
	\SO(\Lambda_0) \to \bZ/2\bZ,\, \alpha \mapsto v_2(\delta(\alpha)) \bmod 2,
\]
which shows that under the hypothesis of the theorem, $\Lambda_1$ and $\Lambda_2$ are indeed $G$-stable. (The argument in \emph{loc.~cit.}~works for $n>1$. For $n=1$, a direct  computation shows that for every $\alpha\in \SO(\Lambda_0)$ we have $\alpha\Lambda_1=\Lambda_1$ and $v_2(\delta(\alpha))\equiv 0$).
\end{proof}

\begin{remark}
It is likely that a result similar to Theorem \ref{thm:evenness} holds over unramified extensions of $\bQ_2$. However, if $K$ is a ramified extension of $\bQ_2$, then the classification of  unimodular $\cO_K$-lattices becomes more complicated (see \cite[\S~93]{OMeara73}), and it is not clear if one should expect such a simple criterion.
\end{remark}


\section{Even unimodular lattices in $\bZ$-bilinear forms over \texorpdfstring{$\bQ_2$}{Q\_2}}

Using Theorem \ref{thm:evenness}, we can now refine Proposition \ref{prop:local-unimodular-existence-at-2} to obtain the existence of an invariant even unimodular lattice.

As in \S~\ref{sec:unimodular-in-local}, let $E_0$ be an \'etale $\bQ_2$-algebra of rank $n$, let $E$ be an \'etale $E_0$-algebra, free of rank $2$ over $E_0$. Denote by $\sigma$ the involution of $E$ fixing $E_0$. Let $\alpha \in E^\times$ satisfy $\alpha \sigma(\alpha)=1$, and denote by $S$ the characteristic polynomial of $\alpha$ over $\bQ_2$.

\begin{proposition}\label{prop:at-2}
Assume that $S(1)$ and $S(-1)$ are non-zero, that $v_2(S(1))$ and $v_2(S(-1))$ are even, and that the class of $(-1)^n S(1)S(-1)$ in $\bQ_2^\times / (\bQ_2^\times)^2$ lies in $\{1,-3\}$. Then there exists a $\lambda \in \mu(E,\sigma)$ and an even unimodular $\bZ_2$-lattice $\Lambda \subset (E,b_\lambda)$ stable under $\alpha$.
\end{proposition}

\begin{proof}
We will use Theorem \ref{thm:evenness}. By Proposition \ref{prop:local-unimodular-existence-at-2}, there is a $\lambda \in \mu(E,\sigma)$ such that the $\bZ$-bilinear form $(E,b_\lambda,\rho_\alpha)$ over $\bQ_2$ satisfies (i). Also, by the Zassenhaus formula, the condition on $v_2(S(-1))$ guarantees that  it satisfies (iii). 

By \cite[Prop.~A.3]{GrossMcMullen02} the bilinear form $(E,b_\lambda)$ has discriminant $1$ or $-3$ in $\bQ_2^\times/(\bQ_2^\times)^2$.

If $E$ is everywhere unramified over $E_0$, then the trace map $\tr\colon E\to E_0$ is surjective, and there is an $e\in E$ with $e+\sigma(e)=1$. It follows that any lattice in $E$ is even, since
\[
	b(x,x)=b((e+\sigma(e)x, x)=b(ex,x)+b(x,ex)=2b(ex,x).
\]
In particular, condition (ii) is satisfied. 

If there is a component $E_{0,w}$ such that $E_w/E_{0,w}$  is ramified, then using Proposition~\ref{prop:HW-of-twisted-form} and Corollary~\ref{cor:suitable-HW} we can modify the component $\lambda_w$ of $\lambda$ to guarantee the existence of an even unimodular lattice in $(E,b_\lambda,\rho_\alpha)$. By Remark \ref{rmk:freedom-of-lambda}, condition (i) in Theorem~\ref{thm:evenness} is not effected.
\end{proof}

\section{Unimodular lattices in $\bZ$-bilinear forms over $\bQ$}\label{sec:the-proof}

\begin{proof}[Proof of Theorem  \ref{bigthm:even-unimodular-charpol}]

Assume given integers $r$ and $s$ satisfying $r\equiv s \bmod 8$, a monic irreducible $P\in \bZ[t]$, and a power $S=P^N$  of degree $r+s$ satisfying \Creciprocal, \Creal, and \Csquares.  If $P$ is linear then $S=(t\pm1)^{r+s}$ and the theorem holds for trivial reasons. So from now on we assume that the degree of $P$ is at least $2$.

\medskip\noindent
\emph{The number fields $E_0$ and $E$}.
Note that \Creciprocal~implies that also $P$ is reciprocal.
Consider the field $F:=\bQ[t]/(P)$, and let $\alpha \in F$ be the image of $t$. Denote by $\sigma$ the involution of $F$ that maps $\alpha$ to $1/\alpha$. Since $\deg P \geq 2$, this involution is non-trivial. Denote by $F_0$ the fixed field of $\sigma_F$. Choose a field extension $E_0$ of $F_0$ of degree $N$, linearly disjoint from $F$. Then  $E:=E_0\otimes_{F_0} F$ is a field, and the characteristic polynomial of $\alpha \in E$ is $S$. Denote by $\sigma$ the canonical extension of $\sigma_F$ to $E$.

In everything what follows, $v$ will denote a place of $\bQ$ and $w$ a place of $E_0$. We will write  $E_{w}$ for the quadratic $E_{0,w}$-algebra $E \otimes_{E_0} E_{0,w}$. 

\medskip\noindent
\emph{Infinite places}.
Let $w$ be an infinite place of $E_0$  and denote by $\alpha_w$  the image of $\alpha$ in $E_{w}^\times$. Then one of the following three hold:
\begin{enumerate}
\item $E_{0,w} = \bR$, $E_w = \bR \times \bR$, and $\alpha_w=(\beta,1/\beta)$ with $\beta \in \bR^\times$ and $|\beta|\neq 1$,
\item $E_{0,w} = \bC$, $E_w = \bC \times \bC$, and $\alpha_w=(\beta,1/\beta)$ with $\beta \in \bC^\times \setminus \bR^\times$ and $|\beta| \neq 1$,
\item $E_{0,w} = \bR$, $E_w =\bC$, and $|\alpha_w| =  1$.
\end{enumerate}

In the first and second case $\mu(E_w,\sigma)$ is trivial, and the bilinear form $(E_w,b_\lambda)$ over $\bR$ has signature $(1,1)$ resp.~$(2,2)$ for all $\lambda \in E_{0,w}^\times$. In the third case $\mu(E_w,\sigma)=\{\pm 1\}$ and $(E_w,b_\lambda)$ has signature $(2,0)$ for $\lambda=1$ and $(0,2)$ for $\lambda=-1$.

As in the introduction to this paper, denote the number of roots $z$ of $S$ with $|z|>1$ by $m=m(S)$. Define
\[
	d_+ := \frac{r-m}{2} , \quad d_- := \frac{s-m}{2}.
\]
Condition \Creal~guarantees that these are non-negative integers. Note that there are exactly $d_+ + d_- = (r + s - 2m)/2$ infinite places of the third type in $E_0$. Choose $\lambda_\infty = (\lambda_w)_{w\mid \infty}$ in
\[
	\mu(\bR\otimes E,\sigma)  = \prod_{w\mid \infty} \mu(E_w,\sigma)
\]
with $d_+$ components $1$ and $d_-$ components $-1$ at places $w$ of the third type. Then $(\bR\otimes E, b_{\lambda_\infty})$ has signature $(r,s)$.

\medskip\noindent
\emph{Finite places}. For every prime number $p$ choose $\lambda_p = (\lambda_w)_{w\mid p}$ in
\[
	\mu(\bQ_p\otimes E, \sigma) = \prod_{w\mid p} \mu(E_w,\sigma)
\]
such that the bilinear form $(\bQ_p\otimes E, b_{\lambda_p})$ over $\bQ_p$ contains an even unimodular $\alpha$-stable $\bZ_p$-lattice. Condition \Csquares~and Propositions \ref{prop:at-p} and \ref{prop:at-2} guarantee that this is possible. Note that for all but finitely many $w$ we have $\lambda_w=1$ in $\mu(E_w,\sigma)$.

\medskip\noindent
\emph{Comparison with a standard lattice}. Let $\Lambda_{r,s}$ be an even unimodular lattice of signature $(r,s)$. The congruence $s\equiv r \bmod 8$ implies $\disc \Lambda_{r,s} = 1$. We claim that for every place $v$ of $\bQ$ we have
\[
	(\bQ_v \otimes E, b_{\lambda_v}) \cong \bQ_v \otimes \Lambda_{r,s},
\]
as bilinear forms over $\bQ_v$. For the infinite place this is clear. For $v=v_p$, note that both forms contain an even unimodular lattice, and that by \cite[Prop.~A.3]{GrossMcMullen02} they have discrimimant  $1$ 
 in $\bQ_p^\times/(\bQ_p^\times)^2$.  This implies that they must be isomorphic. For  $p=2$ this follows from Proposition \ref{prop:2-adic-classification}, and for $p\neq 2$ from \cite[92:1]{OMeara73}.

\medskip\noindent
\emph{Local-global obstruction}.
For every place $v$ of $\bQ$, we now have in $\Br \bQ_v$ the following three elements:
\begin{enumerate}
\item $\epsilon_v(\bQ_v \otimes E, b_{\lambda_v}) = \epsilon_v(\Lambda_{r,s}) $,
\item $\epsilon_v(\bQ_v \otimes E, b_1)$,
\item $\beta(\lambda_v) := \sum_{w|v} \Nm_{E_{0,w}/\bQ_v} \beta_w(\lambda_w)$,
\end{enumerate}
where $\beta_w$ is the map $\mu(E_w,\sigma) \to \Br E_{0,w}$ of Lemma \ref{lemma:mu-to-brauer}. By Proposition \ref{prop:HW-of-twisted-form} these Brauer classes are related by
\[
	\epsilon_v(\Lambda_{r,s}) = \epsilon_v(\bQ_v \otimes E, b_1) + \beta(\lambda_v).
\]
Since $\Lambda_{r,s}$ and $b_1$ are global objects, we have
\[
	\sum_v \inv_v\epsilon_v(\Lambda_{r,s}) = 0,\quad
	\sum_v  \inv_v\epsilon_v(\bQ_v \otimes E, b_1) = 0
\]
in $\bQ/\bZ$, and therefore also $\sum_v \inv_v\beta(\lambda_v) = 0$. Using the commutative square
\[
\begin{tikzcd}
(\Br E_{0,w}) \rar{\inv_w} \dar{\Nm} & \bQ/\bZ \dar{\id} \\
(\Br \bQ_v) \rar{\inv_v} & \bQ/\bZ
\end{tikzcd}
\]
we conclude that $\sum_w \theta_w(\lambda_w) = 0$. By Theorem \ref{thm:local-global} there exists a $\lambda$ in $\mu(E,\sigma)$ specialising to the chosen $\lambda_w$'s. The bilinear form $(E,b_\lambda)$  has signature $(r,s)$, and $(\bQ_p\otimes E,b_\lambda)$ contains an $\alpha$-stable even unimodular $\bZ_p$-lattice $\Lambda_p$ for all $p$.  For all but finitely many $p$ we may take $\Lambda_p=\bZ_p \otimes \cO_E$, and then
\[
	\Lambda := \{ x \in E \mid x \in \Lambda_p \text{ for all $p$} \}
\]
is an $\alpha$-stable even unimodular $\bZ$-lattice in $(E,b_\lambda)$.
\end{proof}

\bigskip

\end{document}